\documentclass[a4paper, 12pt]{amsart}

\usepackage{amsfonts,amsmath,amssymb,amsxtra,amsthm,tikz}
\usepackage{mathrsfs}
\usepackage{times}
\usepackage{anysize}
\usepackage{graphicx}
\marginsize{1in}{1in}{1in}{1in}

\input xy
\xyoption{all}
\usepackage[all,knot]{xy}

\newcommand{\Ext}{\operatorname{Ext}}
\newcommand{\rep}{\operatorname{rep}}
\newcommand{\per}{\operatorname{\mathbf{per}}}

\newcommand{\Hom}{\operatorname{Hom}}
\newcommand{\Iso}{\operatorname{Iso}}

\newcommand{\qis}{\operatorname{qis}}

\newtheorem{theorem}{Theorem}[section]
\newtheorem{corollary}[theorem]{Corollary}
\newtheorem{lemma}[theorem]{Lemma}
\newtheorem{proposition}[theorem]{Proposition}

\theoremstyle{definition}
\newtheorem{definition}[theorem]{Definition}
\newtheorem{example}[theorem]{Example}
\newtheorem{remark}[theorem]{Remark}

\title{Semi-derived and derived Hall algebras for stable categories}
\author{Mikhail Gorsky} 

\address{Universit\'e Paris Diderot -- Paris 7, UFR de
Math\'ematiques, Case 7012, Institut de Math\'ematiques de Jussieu -- Paris Rive Gauche, UMR 7586 du CNRS,
B\^at. Sophie Germain, 75205 Paris Cedex 13, France}
\address{Steklov Mathematical Institute,  
8 Gubkina Street, Moscow, Russia 119991.}
\email{mikhail.gorsky@imj-prg.fr}

\begin{document}

\begin{abstract}
Given a Frobenius category $\mathcal{F}$ satisfying certain finiteness conditions, we consider the localization of its Hall algebra $\mathcal{H(F)}$ at the classes of all projective-injective objects. We call it the {\it ``semi-derived Hall algebra''} $\mathcal{SDH(F, P(F))}.$ We discuss its functoriality properties and show that it is a free module over a twisted group algebra of the Grothendieck group $K_0(\mathcal{P(F)})$ of the full subcategory of projective-injective objects, with a basis parametrized by the isomorphism classes of objects in the stable category $\underline{\mathcal{F}}$. We prove that it is isomorphic to an appropriately twisted tensor product of $\mathbb{Q}K_0(\mathcal{P(F)})$ with the derived Hall algebra (in the sense of To\"{e}n and Xiao-Xu) of $\underline{\mathcal{F}},$ when both of them are well-defined. We discuss some situations where the semi-derived Hall algebra is defined while the derived Hall algebra is not. The main example is the case of $2-$periodic derived category of an abelian category with enough projectives, where the semi-derived Hall algebra was first considered by Bridgeland \cite{Br} who used it to categorify quantum groups. 
\end{abstract}

\maketitle

\section{Introduction}

Hall algebras provide one of the first known examples of additive categorification. They first appeared in works of Steinitz \cite{St} and Hall \cite{Hal} on commutative finite $p-$groups. Later, they reappeared in the work of Ringel \cite{R1} on quantum groups. He introduced the notion of the Hall algebra of an abelian category with finite $\Hom-$ and $\Ext^1-$spaces. This is a vector space over $\mathbb{C}$ with the basis parametrized by the isomorphism classes of objects in the category. The structure constants of the multiplication count in a natural way the (first) extensions with a fixed isomorphism class of an object in the middle. Ringel constructed an isomorphism between the twisted Hall algebra of the category of representations of a simply-laced Dynkin quiver $Q$ over the finite field $\mathbb{F}_q$ and the nilpotent part of the corresponding quantum group, specialized at the square root of $q:$ 
$$U_{\sqrt{q}}(\mathfrak{n}_+) \overset\sim\to \mathcal{H}_{tw}(\rep_{\mathbb{F}_q}(Q)).$$
Later Green \cite{Gr} generalized this result to an arbitrary quiver $Q$ by providing an isomorphism between the nilpotent part of the quantized universal enveloping algebra of the corresponding Kac-Moody algebra and the so-called ``composition'' subalgebra in $\mathcal{H}_{tw}(\rep_{\mathbb{F}_q}(Q))$ generated by the classes of simple objects. Using the Grothendieck group of the category of quiver representations, he introduced an extended version of the Hall algebra which recovers the Borel part of the quantum group. Moreover, Green [loc. cit.] constructed the comultiplication and Xiao \cite{X1} gave the antipode in this twisted extended Hall algebra that make it a self-dual Hopf algebra. Lusztig \cite{L} investigated the geometric version of a composition subalgebra in the Hall algebra $\mathcal{H}_{tw}(\rep_{\mathbb{F}_q}(Q)),$ using perverse sheaves on moduli spaces of quiver representations. This approach led him to the discovery of the canonical basis in $U_{\sqrt{q}}(\mathfrak{n}_+)$ satisfying very pleasant positivity properties. 

Results of Ringel and Green gave rise to a natural question: whether one can realize the whole quantum group $U_{\sqrt{q}}(\mathfrak{g})$ as a certain Hall algebra? It was soon understood that this algebra should be somehow associated to the $2-$periodic, or $\mathbb{Z}/2-$graded, derived category of the abelian category of quiver representations. In this hypothetical construction, two copies of $\rep_{\mathbb{F}_q}(Q)$ should provide two nilpotent parts of the quantum group, while the Cartan part should be recovered from the Grothendieck group. The obstacle was that this $2-$periodic category is not abelian, but rather triangulated. Citing Kapranov \cite{Kap}, ``a
direct mimicking of the Hall algebra construction but with exact triangles replacing exact
sequences, fails to give an associative multiplication''. In other words, the definition of Ringel should be modified in order to provide associative Hall algebras associated, in some way, to triangulated categories.

These ideas motivated several generalizations of Ringel's construction. Peng-Xiao \cite{PX} recovered Lie Kac-Moody algebras from 2-periodic derived categories and, more generally, Hall Lie algebras from 2-periodic triangulated categories. Unfortunately, in their approach, the Cartan part and the rule of its commutation with nilpotent parts appear in a quite ad hoc way. 
Hubery \cite{Hu} proved that the algebra defined in the same way as by Ringel, but for an exact category, is also unital and associative. Kapranov \cite{Kap} introduced a version of the Hall algebra for the bounded derived category of a hereditary abelian category and for its part with cohomologies concentrated in degrees 0 and 1. The latter provided a Heisenberg double of $U_{\sqrt{q}}(\mathfrak{b}_+)$ that is closely related to $U_{\sqrt{q}}(\mathfrak{g})$ but does not coincide with it; in particular, this Heisenberg double does not have a Hopf algebra structure which is a very important property of $U_{\sqrt{q}}(\mathfrak{g}).$ To\"{e}n \cite{T1} gave a construction of what he called {\it derived Hall algebras} for DG-enhanced triangulated categories satisfying certain finiteness conditions. Xiao and Xu showed that this construction provides an associative unital algebra using only the axioms of triangulated categories. Unfortunately, the finiteness conditions one imposes on a category in order to define its derived Hall algebra are quite restrictive: they do hold for bounded derived categories of $\Hom-$finite abelian or exact categories, but they are not satisfied for any periodic triangulated category. Therefore, none of this techniques can give a satisfactory construction of $U_{\sqrt{q}}(\mathfrak{g})$ as a Hall algebra of some kind.

The solution was found by Bridgeland \cite{Br}. He considered the localization of an appropriately twisted Hall algebra of the category of 2-periodic complexes with projective (in $\rep_{\mathbb{F}_q}(Q)$) components at the classes of all contractible complexes. He defined certain reduction of this localization and denoted it by $\mathcal{DH}_{red}(\rep_{\mathbb{F}_q}(Q)).$ He constructed an embedding from $U_t(\mathfrak{g})$ into $\mathcal{DH}_{red}(\mathcal{A}),$ where $\mathcal{A}$ is the category of quiver representations; it is an isomorphism exactly in the Dynkin case. He conjectured that this construction provides the Drinfeld double of the twisted extended Hall algebra $\mathcal{H}_{tw}(\mathcal{A})$ for any hereditary category $\mathcal{A}$ having enough projectives and satisfying natural finiteness conditions. This was proved by Yanagida \cite{Y}.

In this article, we give a waste generalization of Bridgeland's construction and relate it to derived Hall algebras of To\"{e}n. We first notice that the category of 2-periodic complexes with projective components is Frobenius when endowed with a natural exact structure, and its stable category is the 2-periodic homotopy category of the full subcategory $\mathcal{P(A)}$ of projectives in $\mathcal{A}.$ If $\mathcal{A}$ has enough projectives, the latter category is equivalent to the 2-periodic derived category of $\mathcal{A}.$ Thus, Bridgeland's construction can be seen from the following perspective: we have a triangulated category  $\mathcal{T} = \mathcal{D}_{\mathbb{Z}/2}(\mathcal{A})$ for which the derived Hall algebra is not defined, as the finiteness conditions are not satisfied. Then the correct substitute, in some sense, is given by the following rule: one finds a Frobenius category $\mathcal{H}_{\mathbb{Z}/2}(\mathcal{P(A)}),$ whose stable category is equivalent to $\mathcal{T},$ and for which the classical Hall algebra (as of an exact category) is well-defined. Then one takes this Hall algebra and localizes it at the classes of all projective-injective objects. 

Now we consider an arbitrary Frobenius category $\mathcal{F}$ satisfying some finiteness conditions. We define the {\it semi-derived Hall algebra} $\mathcal{SDH(F, P(F))}$ as the localization of $\mathcal{H(F)}$ at the classes of all projective-injective objects. We prove that it is functorial under fully faithful maps of Frobenius categories. We show that it is a free module over the group algebra of the Grothendieck group of the full subcategory $\mathcal{P(F)}$ of projective-injective objects twisted by the Euler form. Any choice of representatives in $\mathcal{F}$ of the isomorphism classes of objects in $\underline{\mathcal{F}}$ yields a basis of this module. Using this property, we prove that $\mathcal{SDH(F, P(F))}$ with an appropriately twisted multiplication is isomorphic to the tensor product of the derived Hall algebra of $\underline{\mathcal{F}}$ with the group algebra of $K_0(\mathcal{P(F)}),$ when the latter is well-defined. Bridgeland's construction demonstrates that in some situations $\mathcal{SDH(F, P(F))}$ is well-defined while the derived Hall algebra of $\underline{\mathcal{F}}$ and the twist are not. Triangulated categories which are equivalent to stable categories of some Frobenius categories are called {\it algebraic}. Conceptually, all the reasonable triangulated categories appearing in algebra and geometry are algebraic. 

Throughout this paper we work with categories linear over finite fields. There are numerous variations and generalizations of Hall algebras for categories linear over $\mathbb{C}.$ One can consider classical Hall algebras, their geometric versions \`{a} la Lusztig, motivic or cohomological Hall algebras. We do not deal with them in this paper. Nonetheless, we strongly believe that our approach can be generalized to some of these frameworks. As a first example of such a kind, one can consider a geometric analogue of Bridgleand's algebras for simply-laced Dynkin quivers introduced by Qin in his recent work \cite{Qin}. We leave the further investigation for future research.
%This will be a subject of future work.

The paper is organized as follows. In section 2, we recall the notions of Hall and derived Hall algebras, Euler forms and Frobenius categories. In section 3, we introduce the semi-derived Hall algebras. We prove their functoriality properties. We show that they are free over the quantum tori of projective-injectives. We prove that they are invariant under a big class of equivalences of stable categories. In section 4, we relate our semi-derived Hall algebras to the derived Hall algebras of To\"{e}n and Xiao-Xu. In section 5, we briefly discuss some examples. In section 6, we outline some directions for future work.

This is a part of my ongoing Ph.D. project at the Universit\'{e} Paris 7 under the joint supervision of Prof. Bernhard Keller and Prof. Olivier Schiffmann. I am very grateful to both of them for their support, patience and valuable comments. I am also grateful to Tom Bridgeland, Mikhail Kapranov and Fan Qin for useful conversations. The work was supported by DIM RDM-IdF of the R\'{e}gion \^{I}le-de-France.

\section{Preliminaries}

\subsection{Hall algebras}

Let $\mathcal{E}$ be an essentially small exact category (in the sense of Quillen\cite{Q}), linear over a finite field $k.$ For the basics on exact categories, we refer to \cite{Buh}\cite{Kel2} and references therein. Assume that $\mathcal{E}$ has finite morphism and (first) extension spaces:
$$|\mbox{Hom}(A, B)| < \infty, \quad  |\Ext^1 (A, B)| < \infty, \quad \forall A, B \in \mathcal{E}.$$
For a triple of objects $A, B, C \in \mathcal{E},$ we denote by $\Ext^1 (A, C)_B \subset \Ext^1 (A, C)$  the subset parameterizing extensions whose middle term is isomorphic to B. The {\it Hall}, or {\it Ringel-Hall algebra} $\mathcal{H(E)}$ is the $\mathbb{Q}-$vector space whose basis is formed by the isomorphism classes $[A]$ of objects $A$ of $\mathcal{E},$
with the multiplication given by
$$[A] \diamond [C] = \sum\limits_{B \in \mbox{\footnotesize{Iso}}(\mathcal{E})} \frac{|\Ext^1_{\mathcal{E}} (A, C)_B|}{|\mbox{Hom}_{\mathcal{E}} (A, C)|} [B].$$

The following result as well as the definition of $\mathcal{H(E)}$ is due to Ringel \cite{R1} for an abelian $\mathcal{E}$; Hubery \cite{Hu} generalized this to the case of exact $\mathcal{E}$.

\begin{theorem}
The algebra $\mathcal{H(E)}$ is associative and unital. The unit is given by $[0]$, where $0$ is the zero object of $\mathcal{E}$.
\end{theorem}

\begin{remark}
The choice of the structure constants
$\frac{|\Ext^1_{\mathcal{E}}(A,B)_{C}|}{|\Hom_{\mathcal{E}}(A,B)|} $ is the same as in \cite{Br} and the most convenient for our calculations.
This choice is equivalent to that of the usual structure constants
$|\{ B' \subset C | B' \cong B, C/B' \cong A  \}|,$
called the {\it Hall numbers} and appearing in \cite{R1},\cite{Sch} and \cite{Hu}.
\end{remark}

\subsection{Euler form and twists}
Assume that $\mathcal{E}$ is locally homologically finite and that all higher extension spaces are finite:
$$\forall A, B \in \mathcal{E} \quad \exists p_0: \quad \Ext^p (A, B) = 0, \quad \forall p > p_0;$$
$$|\Ext^p (A, B)| < \infty, \quad \forall p \geq 0, \quad \forall A, B \in \mathcal{E}.$$
For objects $A, B \in \mathcal{E}$, we define the Euler form by the formula
$$\left\langle A, B \right\rangle := \prod_{i \in \mathbb{Z}} |\Ext^i_{\mathcal{E}}(A,B)|^{{(-1)}^i}.$$
It is well known that (thanks to the five-lemma) this form descends to a bilinear form
on the Grothendieck group $K_0(\mathcal{E})$ of $\mathcal{E}$,
denoted by the same symbol:
$$\left\langle \cdot, \cdot \right\rangle:  K_0(\mathcal{E}) \times K_0(\mathcal{E}) \to \mathbb{Q}^{\times}.$$

The {\it twisted Hall algebra} $\mathcal{H}_{tw}(\mathcal{E})$ 
is the same vector space as $\mathcal{H}(\mathcal{E})$  
with the twisted multiplication
\begin{equation}
[A] * [B] := \sqrt{\left\langle A, B \right\rangle} \cdot [A] \diamond [B], \quad \forall  A, B \in \Iso(\mathcal{E}).
\end{equation}

\subsection{Derived Hall algebras}
In \cite{T1}, To\"{e}n defined a version of Hall algebras for certain class of dg-enhanced triangulated categories. He called these objects {\it derived Hall algebras}. This work was further generalized by Xiao and Xu \cite{XX1} for triangulated categories without assumptions on the existence of a dg-enhancement. The construction is as follows.
Let $\mathcal{T}$ be an essentially small triangulated category, linear over a finite field $k.$ We also assume that $\mathcal{T}$ has finite morphism spaces. We denote the shift functor in $\mathcal{T}$ by $\Sigma.$ As usually, the space of $i-$th extensions of $X$ by $Y$ is defined as $\Ext^i_{\mathcal{T}}(X,Y) = \Hom_{\mathcal{T}}(X, \Sigma^i Y),$ for $X, Y \in \mathcal{T}$ and for $i \in \mathbb{Z}.$ 
Assume that $\mathcal{T}$ is {\it left locally homologically finite}, that is, it satisfies the following property: for each pair of objects $A, B \in \mathcal{F},$ there exists $N \in \mathbb{N},$ such that for each $i > N,$ we have
$$|\Ext^{-i} (A, B)| = 0.$$

%Since we assume $\mathcal{T}$ to be $\Hom-$finite, this is equivalent to the following condition:

%\begin{itemize}
%\item[(\underline{LLHF'})] For each pair of objects $A, B \in \mathcal{F},$ we have
%$$\prod\limits_{i \geq 0} |\Ext^{-i} (A, B)| < \infty.$$
%\end{itemize}

The derived Hall algebra $\mathcal{DH}(\mathcal{T})$ is the $\mathbb{Q}-$vector space whose basis is formed by the isomorphism classes $[A]$ of objects $A$ of $\mathcal{T},$
with the multiplication defined by
$$[A] \diamond [C] = \sum\limits_{B \in \mbox{\footnotesize{Iso}}(\mathcal{T})} \frac{|\Ext^1_{\mathcal{T}} (A, C)_B|}{|\mbox{Hom}_{\mathcal{T}} (A, C)|} \prod\limits_{i > 0}|\Ext^{-i}_{\mathcal{T}}(A, C)|^{(-1)^{(i-1)}} [B].$$
Here the set $\Ext^1_{\mathcal{T}} (A, C)_B$ is defined as in the exact case. 

\begin{theorem} \cite{T1}\cite{XX1}\cite{XX2}
The algebra $\mathcal{DH(T)}$ is associative and unital. The unit is given by $[0]$, where $0$ is the zero object of $\mathcal{T}$.
\end{theorem}

As in the subsection 2.1, our choice of structure constants is not the one given by To\"{e}n but an equivalent and slightly renormalized one. It is due to Kontsevich-Soibelman \cite{KS}, cf. also \cite{XX2}. 

\subsection{Frobenius and stable categories}

%The concept of Frobenius categories goes back to works of Alex Heller \cite{Hel}, but it was Happel wh
Recall that in an exact category $\mathcal{F}$ an object $P$ is called projective, if $\Ext^1_\mathcal{F}(P,X) = 0,$ for any object $X \in \mathcal{F}.$ Injective objects are defined in the dual way. An exact category $\mathcal{F}$ is {\it Frobenius}, if it has enough projectives and enough injectives and if, moreover, an object is projective if and only if it is injective. Let $\mathcal{F}$ be a Frobenius category, $\mathcal{P(F)}$ the full subcategory of projective-injective objects in $\mathcal{F}.$ We define the {\it stable category} $\underline{\mathcal{F}}$ of $\mathcal{F}.$ The objects
of $\underline{\mathcal{F}}$ are the same as the objects of $\mathcal{F},$ and the morphism spaces $\Hom_{\underline{\mathcal{F}}}(X, Y)$ are the
morphism spaces in $\mathcal{F}$ modulo morphisms factoring through projective-injective
objects. The stable category $\underline{\mathcal{F}}$ is a triangulated category in a natural way \cite{Hap}.
The shift is inverse to the auto-equivalence of $\underline{\mathcal{F}}$ that sends an object to the kernel of its projective cover (the latter is called the relative syzygy functor). 

A triangulated category $\mathcal{T}$ is called {\it algebraic} if it is equivalent to the stable category of a Frobenius category $\mathcal{F}.$ In this case, one says also that $\mathcal{F}$ is a {\it Frobenius model} for $\mathcal{T}.$

\section{Semi-derived Hall algebras for stable categories}

Assume that a Frobenius category $\mathcal{F}$ satisfies the following conditions:

\begin{itemize}
\item[(C1)] $\mathcal{F}$ is essentially small, idempotent complete and linear over some ground field $k;$
\item[(C2)] $\mathcal{F}$ is {\it $\Hom-$finite}. That is, for each pair of objects $A, B \in \mathcal{F},$ we have
$$|\Hom (A, B)| < \infty.$$
%\item[(C3)] For each pair of objects $A, B \in \mbox{Ob}(\mathcal{F}),$ there exists $N > 0$ such that for all $p > N,$ we have
%$$\Ext^p (A, B) = 0.$$
\end{itemize}

Note that these assumptions ensure that $\mathcal{F}$ is Krull-Schmidt. Moreover, it is known that if $\mathcal{F}$ is Krull-Schmidt, then its stable category $\underline{\mathcal{F}}$ is also Krull-Schmidt. We prove the following lemma similar to \cite[2.3]{Hap}.

\begin{lemma} \label{extstable}
All extension spaces in the category $\mathcal{F}$ coincide with those in the stable category $\underline{\mathcal{F}}.$ Explicitly, for any $M, N \in \mathcal{F},$ the canonical map:
$$\Ext^p_{\mathcal{F}} (M, N) \to \Ext^p_{\underline{\mathcal{F}}} (M, N)$$
is bijective for all $p > 0.$
\end{lemma}

\begin{proof}
By definition of the triangulated structure on the stable category $\underline{\mathcal{F}},$ we have a family of conflations
$$\Sigma^{-i} M \rightarrowtail P(\Sigma^{-i-1}M) \twoheadrightarrow \Sigma^{-i-1}M, \quad\quad i \in \mathbb{Z},$$
where $P(\Sigma^{-i-1}M)$ belongs to $\mathcal{P(F)}, \quad \Sigma$ is the suspension functor in $\underline{\mathcal{F}}.$ Thus, the complex
$$R(M)^{\bullet} = \ldots \to P(\Sigma^{-2}M) \to P(\Sigma^{-1}M) \to P(M)  \to 0$$
is a projective resolution of $M$ in $\mathcal{F}.$ Therefore, extensions of $M$ by $N$ are exactly the homologies of the complex $\Hom_{\mathcal{F}}(R(M)^{\bullet},N),$ i.e. $\Ext^p_{\mathcal{F}} (M, N)$ is the quotient of the set of morphisms $\Hom_{\mathcal{F}}({\Sigma^{-(p)} M}, N)$ by the subset of morphisms which factor through $P(\Sigma^{-p-1}M).$ This last subset is, by definition, the full subset of $\Hom_{\mathcal{F}}({\Sigma^{-(p)} M},N)$ containing morphisms which factor through a projective-injective object. Thus, we have an isomorphism 
$$\Ext^p_{\mathcal{F}} (M, N) \overset\sim\to \Hom_{\underline{\mathcal{F}}}({\Sigma^{-(p)} M},N).$$
The right hand side is nothing but $\Hom_{\underline{\mathcal{F}}}({M}, \Sigma^{p}  N),$ that is equal to $\Ext^p_{\underline{\mathcal{F}}}(M,  N).$
\end{proof}

\begin{corollary}
If a Frobenius category $\mathcal{F}$ satisfies condition (C2), it is also $\Ext^p-$finite, for any $p > 0.$ That is, for any $M, N \in \mathcal{F}$ and for any $p > 0,$ we have 
$$|\Ext^p_{\mathcal{F}}(M,N)| < \infty.$$
\end{corollary}

\begin{proof}
By (the proof of) Lemma \ref{extstable}, we know that the set
$\Ext^p_{\mathcal{F}} (M, N)$ is a subquotient of the set $\Hom_{\mathcal{F}}({\Sigma^{-(p)} M},N).$ The latter being finite, so is the former.
\end{proof}

It follows that the classical Hall algebra $\mathcal{H(C)}$ of the exact category $\mathcal{F}$ satisfying conditions (C1) and (C2) is well-defined.

Following \cite{Schl06}, we define {\it a map between Frobenius categories} to be an exact functor between them sending projective-injective objects to projective-injective ones. 
Such a map $F: \mathcal{F}' \to \mathcal{F}$ induces an exact functor between the stable categories $\underline{F}: \underline{\mathcal{F}'} \to \underline{\mathcal{F}}$ \cite[2.6]{Hap}. If $F$ is fully faithful then $\underline{F}$ is fully faithful as well, see, e.g., \cite[Remark~7]{Kuen13}. Following \cite{Sch}, we say that an exact functor $G: \mathcal{A} \to \mathcal{B}$ between exact categories is {\it extremely faithful}, if it induces isomorphisms $\Ext^i(M,N) \overset\sim\to \Ext^i(G(M), G(N)),$ for all $M, N \in \mathcal{A}$ and all $i \geq 0.$
By Lemma \ref{extstable}, extensions in Frobenius categories are certain morphisms in their stable categories. Therefore, we can make the following observation. 

\begin{lemma} \label{extreme}
Any fully faithful map $F: \mathcal{F}' \to \mathcal{F}$ between Frobenius categories is extremely faithful.
\end{lemma}

Let $F: \mathcal{F}' \to \mathcal{F}$ be an exact functor. It induces a natural linear map of vector spaces 
$$F_{\ast}: \mathcal{H}(\mathcal{F}') \to \mathcal{H(F)},\quad  [M] \mapsto [F(M)].$$
In general, it is not a morphism of algebras. Nonetheless, by \cite[Corollary 1.16]{Sch} and Lemma \ref{extreme}, we get the following result about the functoriality of Hall algebras.

\begin{corollary} \label{functfrob}
If $F: \mathcal{F}' \to \mathcal{F}$ is a fully faithful map between essentially small $\Hom-$finite Frobenius categories, then $F_{\ast}$ is an embedding of algebras.
\end{corollary}

Since the exact structure on the category $\mathcal{P(F)}$ of projective-injectives in $\mathcal{F}$ splits, the Euler form on $K_0(\mathcal{P(F)}) = K_0^{split}(\mathcal{P(F)})$ is well-defined and is given on classes of objects simply as the cardinality of morphism spaces:

$$\left\langle\cdot,\cdot\right\rangle: K_0(\mathcal{P(F)}) \times K_0(\mathcal{P(F)}) \to \mathbb{Q}^{\times}, \quad \left\langle A, B \right\rangle = |\Hom(A, B)|, \qquad \mbox{for} \quad A, B \in \mathcal{P(F)}.$$

Similarly, the Euler form is well-defined and given by the same formula on $K_0(\mathcal{P(F)}) \times K_0(\mathcal{F})$ and on $K_0(\mathcal{F}) \times K_0(\mathcal{P(F)}).$
We define the {\it quantum torus of projective-injectives} $\mathbb{T}(\mathcal{P(F)})$ as the group algebra of $K_0(\mathcal{P(F)})$ with the multiplication twisted by the Euler form.
For any $P \in \mathcal{P(F)}, C \in \mathcal{F},$  their products in the Hall algebra $\mathcal{H(F)}$ take very simple form: 
$$[P] \diamond [C] = \frac{1}{|\Hom(P, C)|} [P \oplus C] = \frac{1}{\left\langle P, C  \right\rangle}[P \oplus C];$$
$$[C] \diamond [P] = \frac{1}{|\Hom(C, P)|} [P \oplus C] = \frac{1}{\left\langle C, P  \right\rangle}[P \oplus C].$$
It follows that the set of all classes of the form $[P],$ for $P \in \mathcal{P(F)},$ satisfies the Ore conditions in this algebra. Therefore, we can localize $\mathcal{H(F)}$ at these classes.

\begin{definition}
The {\it semi-derived Hall algebra of the pair} $(\mathcal{F}, \mathcal{P(F)})$ is the localization of $\mathcal{H(C)}$ at the classes of all projective-injective objects:
$$\mathcal{SDH(F, P(F))} := \mathcal{H(F)}[[P]^{-1}| P \in \mathcal{P(F)}].$$
\end{definition}

By definition, $\mathcal{SDH(F, P(F))}$ is an associative unital algebra, where the unit is given by $[0]$, $0$ being the zero object of $\mathcal{F}$. Moreover, by its definition, it satisfies the following functoriality property.

\begin{proposition}
If $F: \mathcal{F}' \to \mathcal{F}$ is a fully faithful map between essentially small $\Hom-$finite Frobenius categories, then $F_{\ast}$ induces an embedding of algebras $\mathcal{SDH(F', P(F'))} \hookrightarrow \mathcal{SDH(F, P(F))}.$
\end{proposition}

We have  natural left and right actions of $\mathbb{T}(\mathcal{P(F)})$ on $\mathcal{SDH(F, P(F))}$ given by the Hall product. Let us denote by $\mathcal{M(F)}$ this bimodule structure on $\mathcal{SDH(F, P(F))}.$

\begin{theorem} \label{frobfree}
Assume that $\mathcal{F}$ satisfies conditions (C1) and (C2). Then $\mathcal{M(F)}$ is a free right (resp. left) module over $\mathbb{T}(\mathcal{P(F)}).$ Each choice of representatives in $\mathcal{F}$ of the isomorphism classes of the stable category $\underline{\mathcal{F}}$ yields a basis.
\end{theorem}

\begin{proof}
Assume that the images in $\underline{\mathcal{F}}$ of two objects $M, M'$ from $\mathcal{F}$ are isomorphic. Since the image $\underline{\mathcal{P(F)}}$ of $\mathcal{P(F)}$ in $\underline{\mathcal{F}}$ is contained in the isomorphism class of $0,$ we have 
$$\underline{\mathcal{F}} = \underline{\mathcal{F}}/\underline{\mathcal{P(F)}}.$$
This means that there is a sequence of objects $M_0 = M, M_1, M_2,\ldots, M_n = M'$ in $\mathcal{F},$ such that for each $i = 1, 2, \ldots, n$ there is either a conflation 
$$P \rightarrowtail M_{i-1} \stackrel{\qis}{\twoheadrightarrow} M_i,$$
or a conflation
$$P \rightarrowtail M_i \stackrel{\qis}{\twoheadrightarrow} M_{i-1},$$
with $P$ projective. Therefore, we either have
$$[M_i] = [P \oplus M_{i-1}] = |\Hom(P, M_{i-1})| [P] \diamond [M_{i-1}],$$
or
$$[M_{i-1}] = [P \oplus M_i] = |\Hom(K, M_i)| [K] \diamond [M_i] \Rightarrow [M_i] = \frac{1}{|\Hom(K, M_i)|} [K]^{-1} \diamond [M_{i-1}].$$
It follows that $[M'] \in \mathbb{T}(\mathcal{P(F)}) \diamond [M].$ Therefore, (the representatives of) the isomorphism classes in the stable category $\underline{\mathcal{F}}$ generate $\mathcal{M}(\mathcal{F})$ over $\mathbb{T}(\mathcal{P(F)}).$ It remains to prove that they are independent over this quantum torus.

One can decompose $\mathcal{M(F)}$ into the direct sum 
$$\mathcal{M(F)} = \bigoplus\limits_{\alpha \in \Iso(\underline{\mathcal{F}})} \mathcal{M}_{\alpha} \mathcal{(F)},$$
where $\mathcal{M}_{\alpha} \mathcal{(F)}$ is the component containing the classes of all objects whose isomorphism class in $\underline{\mathcal{F}}$ is $\alpha.$ We claim that for each $\alpha,$ the $\mathbb{T}(\mathcal{P(F)})-$submodule $\mathcal{M}_{\alpha}(\mathcal{F})$ is free of rank one. Let $M$ be an object of $\mathcal{F}.$ By the above argument, the map
\begin{equation} \label{maptm}
\mathbb{T}(\mathcal{P(F)}) \to \mathcal{M}_{[M]}(\mathcal{F}), \quad [K] \mapsto [K] \diamond [M]
\end{equation}
is surjective. Since $\mathbb{T}(\mathcal{P(F)})$ is the (twisted) group algebra of $K_0(\mathcal{P(F)}),$  Lemma \ref{k0splitinj} below shows that its composition with the natural map 
$$\mathcal{M}_{[M]}(\mathcal{F}) \to \mathcal{M}(\mathcal{F}) \to \mathbb{Q}[K_0^{split}(\mathcal{F})]$$
is injective. Here, the last map is the identity on objects; it is well-defined, since $\mathcal{M}(\mathcal{F})$ has a natural grading by the group $K_0^{split}(\mathcal{F}).$ Therefore, the map (\ref{maptm}) is bijective, q.e.d.
\end{proof}

\begin{lemma} \label{k0splitinj}
Under conditions of Theorem \ref{frobfree}, the natural map 
$$i: K_0(\mathcal{P(F)}) \rightarrow K_0^{split}(\mathcal{F}), \quad [M] \mapsto [M]$$
is injective.
\end{lemma}

\begin{proof}
Since $\mathcal{F}$ is Krull-Schmidt, one can define the ``projective part'' of an object in $\mathcal{E}$: each object $M \in \mathcal{F}$ can be decomposed in a unique way (up to a permutation of factors) into a finite direct sum of indecomposables:
$$M = \bigoplus\limits_{i=1}^{m(M)} M_i \oplus \bigoplus\limits_{j=1}^{k(M)} M_j',$$
where all $M_i$ belong to $\mathcal{P(F)}$ while the $M_j'$ do not. Then 
$$\phi: K_0^{split}(\mathcal{F}) \twoheadrightarrow K_0^{split}(\mathcal{P(F)}) = K_0(\mathcal{P(F)}), [M] \mapsto \bigoplus\limits_{i=1}^{m(M)} M_i$$
is a well-defined group epimorphism, and we get 
$$\phi \circ i = \mbox{Id}_{K_0(\mathcal{P(F)})}.$$
\end{proof}

%\end{proof}

\begin{theorem} \label{isomorphism1}
Let $\mathcal{F}', \mathcal{F}$ be two Frobenius categories satisfying assumptions (C1) and (C2), $\mathcal{P}', \mathcal{P}$ their full subcategories of projective-injective objects. Assume that 
$$F: \mathcal{F}' \to \mathcal{F}$$
is a fully faithful exact functor inducing an equivalence of the stable categories 
$$\underline{\mathcal{F}'} \overset{F}{\overset\sim\to} \underline{\mathcal{F}}$$
and an isomorphism of the Grothendieck groups of projective-injectives
\begin{equation} \label{k0p_iso}
K_0(\mathcal{P}(\mathcal{F}')) \overset{F}{\overset\sim\to} K_0(\mathcal{P(F)}).
\end{equation}
Then $F$ induces an isomorphism of algebras
$$F_\ast: \mathcal{SDH(F', P(F'))} {\overset\sim\to} \mathcal{SDH(F, P(F))}.$$
\end{theorem}

\begin{proof}
Since $F$ is fully faithful and induces the isomorphism (\ref{k0p_iso}), it also induces an isomorphism of the quantum tori of projective-injectives:
$$\mathbb{T}(\mathcal{P(F')}) \overset{F}{\overset\sim\to} \mathbb{T}(\mathcal{P(F)}).$$
Therefore, by Theorem~\ref{frobfree}, $F$ induces an isomorphism 
$$\mathcal{M(F')} {\overset\sim\to} \mathcal{M(F)}$$ 
of the free modules over isomorphic quantum tori with bases which are in bijection by the stable equivalence. By the full faithfulness and Lemma~\ref{extstable}, the multiplication is preserved as well, i.e. $F$ induces the desired isomorphism of algebras.
\end{proof}

%NB. Functoriality? 

\section{Semi-derived vs. derived Hall algebras}

Let $\mathcal{F}$ be a Frobenius category satisfying conditions (C1), (C2). As before, we denote by $\underline{\mathcal{F}}$ its stable category. It is evidently $\Hom-$finite, since its morphism spaces are subquotients of the morphism spaces in $\mathcal{F}.$ Assume that $\underline{\mathcal{F}}$ is left locally homologically finite.

We introduce the {\it relative Euler form} 
$$\left\langle \cdot, \cdot \right\rangle_{(\mathcal{F},\underline{\mathcal{F}})}: K_0(\mathcal{F}) \times K_0(\mathcal{F}) \to \mathbb{Q}$$
by the following rule: for each pair $A, B \in \mathcal{F},$ we pose
$$\left\langle A, B \right\rangle_{(\mathcal{F},\underline{\mathcal{F}})} = \frac{|Hom_{\mathcal{F}}(A, B)|}{|Hom_{\underline{\mathcal{F}}}(A, B)} \cdot \prod\limits_{i > 0}|\Ext^{-i}_{\underline{\mathcal{F}}}(A, C)|^{(-1)^{(i-1)}}.$$

\begin{lemma} \label{releuler}
The form $$\left\langle \cdot, \cdot \right\rangle_{(\mathcal{F},\underline{\mathcal{F}})}$$ is a well-defined group homomorphism.
\end{lemma}

\begin{proof}
The statement follows from Lemma \ref{extstable} and the comparison of long exact sequences of extensions in $\mathcal{F}$ and $\underline{\mathcal{F}}.$ 
Consider an arbitrary conflation 
$$A_1 \rightarrowtail A_2 \twoheadrightarrow A_3$$ 
in the category $\mathcal{F}.$ For any $B \in \mathcal{F},$ we have two long exact sequences of extensions of elements of this conflation by $B$: the sequence of extensions in $\mathcal{F}$ and the one of those in $\underline{\mathcal{F}}:$
\begin{equation} \label{extfrobenius}
0 \to \Hom_{\mathcal{F}}(A_3, B) \to \Hom_{\mathcal{F}}(A_2, B) \to \Hom_{\mathcal{F}}(A_1, B) \overset{f}\to \atop  \Ext^1_{\mathcal{F}}(A_3, B) \to \Ext^1_{\mathcal{F}}(A_2, B) \to \Ext^1_{\mathcal{F}}(A_1, B) \to \ldots;
\end{equation}
$$\ldots \Ext^{-1}_{\underline{\mathcal{F}}}(A_3, B) \to \Ext^{-1}_{\underline{\mathcal{F}}}(A_2, B) \to \Ext^{-1}_{\underline{\mathcal{F}}}(A_1, B) \overset{g}\to$$
\begin{equation} \label{longexactstable}
 \Hom_{\underline{\mathcal{F}}}(A_3, B) \to \Hom_{\underline{\mathcal{F}}}(A_2, B) \to \Hom_{\underline{\mathcal{F}}}(A_1, B) \to
\end{equation}
$$\Ext^1_{\underline{\mathcal{F}}}(A_3, B) \to \Ext^1_{\underline{\mathcal{F}}}(A_3, B) \to \Ext^1_{\underline{\mathcal{F}}}(A_3, B) \ldots.$$

By lemma \ref{extstable}, any term of the form $\Ext^i_{\mathcal{F}} (A_j, B),$ with $i > 0, j = 1,2,3,$ is isomorphic to its counterpart in the second sequence. Thus, we have an isomorphism $Ker(f) \overset\sim\to Ker(g).$ We have two exact sequences to the left of $Ker(f)$ in (\ref{extfrobenius}), respectively of $Ker(g)$ in (\ref{longexactstable}), and find out that the alternating products of their terms both equal $1.$ Hence their quotient equals $1$ as well. On the other hand, it coincides with 
$$\left\langle A_1, B \right\rangle_{(\mathcal{F},\underline{\mathcal{F}})} \cdot \frac{1}{\left\langle A_2, B \right\rangle_{(\mathcal{F},\underline{\mathcal{F}})}} \cdot \left\langle A_3, B \right\rangle_{(\mathcal{F},\underline{\mathcal{F}})}.$$
The statement follows from this and from the dual result concerning $\left\langle B, A_i \right\rangle_{(\mathcal{F},\underline{\mathcal{F}})},$ which has a similar proof.
\end{proof}

It is easy to see that, by Lemma \ref{releuler}, one can twist the multiplication in $\mathcal{SDH}(\mathcal{F, P(F)})$ by the rule 
$$A \ast B := \left\langle A, B \right\rangle_{(\mathcal{F},\underline{\mathcal{F}})} A \diamond B.$$
We call the result {\it the twisted semi-derived Hall algebra} $\mathcal{SDH}(\mathcal{F, P(F)})_{tw}.$ 

We are ready now to present the main result of this section comparing our construction with derived Hall algebras.

\begin{theorem} \label{semiderder}
Assume that a Frobenius category $\mathcal{F}$ satisfies properties (C1) and (C2), and its stable category $\underline{\mathcal{F}}$ is left locally homologically finite. Then each choice of representatives in $\mathcal{F}$ of the isomorphism classes of the stable category $\underline{\mathcal{F}}$ yields an isomorphism
$$\mathcal{SDH}(\mathcal{F, P(F)})_{tw} \overset\sim\to (\mathcal{DH}(\underline{\mathcal{F}}) \otimes \mathbb{Q}[K_0(\mathcal{P(F)})]).$$
\end{theorem}

\begin{proof}
The right-hand side is a free module over  $\mathbb{Q}[K_0(\mathcal{P(F)})],$ with the basis parameterized by the isomorphism classes of objects in $\underline{\mathcal{F}}.$ By Theorem \ref{frobfree} and by the choice of the twist, the left-hand side is also a free module over $\mathbb{Q}[K_0(\mathcal{P(F)})].$ Moreover, each choice of representatives in $\mathcal{F}$ of the isomorphism classes of the stable category $\underline{\mathcal{F}}$ yields a basis of this module. The group algebra action is the same on both sides. It remains to show that the multiplicative structures are the same on both sides. By lemma \ref{extstable}, the sets $\Ext^1_{\mathcal{F}}(A,B)_{C}$ and $\Ext^1_{\underline{\mathcal{F}}}(A,B)_{C}$ are isomorphic for any triple of objects $A, B, C \in \mathcal{F}.$ Now the statement follows from the form of the structure constants in $\mathcal{SDH}(\mathcal{F, P(F)})$ and in $\mathcal{DH}(\underline{\mathcal{F}}),$ by the choice of the twist.
\end{proof}
\section{Examples}

\begin{example}
$\mathcal{F}$ is the category of bounded (or $m$-periodic complexes) over an exact category $\mathcal{E},$ with the component-wise split exact structure. This is a Frobenius category, whose stable category is the bounded (resp. $m-$periodic) homotopy category $\mathcal{H}^b(\mathcal{E})$ (res. $\mathcal{H}_{\mathbb{Z}/m}(\mathcal{E})$). If $\mathcal{E}$ satisfies conditions (C1) and (C2), so do $\mathcal{H}^b(\mathcal{E})$ and $\mathcal{H}_{\mathbb{Z}/m}(\mathcal{E}).$ There, our construction provides a Hall-like algebra for bounded and periodic homotopy categories.
\end{example}

\begin{example} 
$\mathcal{F} = \mathcal{C}^b(\mathcal{P(E)}),$ where $\mathcal{P(E)}$ is the full subcategory of projective objects in an exact category $\mathcal{E}$ with enough projectives and where each object has finite projective resolution. Then the stable category $\underline{\mathcal{F}}$ is equivalent to the bounded derived category $\mathcal{D}^b(\mathcal{E}).$ Therefore, our construction provides a version of the Hall algebra for the bounded derived category $\mathcal{D}^b(\mathcal{E}).$ It was introduced by the author in \cite{Gor} and called the ``semi-derived Hall algebra of $\mathcal{E}$''. See [loc. cit.] for the detailed treatment.
\end{example}

\begin{example}
$\mathcal{F}$ is the category of $m-$periodic complexes over $\mathcal{P(E)},$ for $m > 1.$  As in the previous example, the stable category is equivalent the $m-$periodic derived category $\mathcal{D}_{\mathbb{Z}/m}(\mathcal{E}).$ As $\mathcal{D}_{\mathbb{Z}/m}(\mathcal{E})$ is an $m-$periodic triangulated category, it is never left locally homologically finite. Therefore, one cannot define its derived Hall algebra, and neither the twist nor the right-hand side in the identity in Theorem \ref{semiderder} are well-defined. On the other hand, the construction presented in this work provides an associative algebra. The case of $m = 2$ was first considered in the work of Bridgeland \cite{Br} that provided the main inspiration to our work, see also \cite{Gor}. Yanagida \cite{Y} proved the conjecture of Bridgeland \cite{Br} that, under certain conditions and for $\mathcal{E}$ abelian, this algebra with an appropriate twist provides the Drinfeld double of the twisted extended Hall algebra of $\mathcal{E}$. The generic version of such an algebra (in the abelian case, but for an arbitrary positive $m$) was introduced in \cite{CD}. Zhao \cite{Z} proved that for $\mathcal{E}$ abelian, the category  $\mathcal{D}_{\mathbb{Z}/m}(\mathcal{E})$ is equivalent to the {\it generalized root category} of $\mathcal{E},$ see references in [loc. cit]. Therefore, the algebra (that we call in \cite{Gor} ``the $\mathbb{Z}/m-$graded semi-derived Hall algebra of $\mathcal{E}$'')  is also the substitute of the non-defined derived Hall algebra of the root category.
\end{example}

\begin{example}
Dually, we can take as $\mathcal{F}$ the category of bounded or periodic complexes over the full subcategory of injectives in an exact category $\mathcal{E}$ with enough injectives and where each object has finite injective resolution. The corresponding stable category is again equivalent to the bounded (resp. periodic) derived category of $\mathcal{E},$ so we get an algebra isomorphic to the one from previous examples.
\end{example}

\begin{example}
Let $\mathcal{D}$ be a differential graded (DG) category. We freely use basic facts on DG-categories that can be found, e.g., in surveys \cite{Kel3} and \cite{T2}. One can define the category  $\mathcal{C}_{dg}(\mathcal{D})$ of DG-modules over $\mathcal{D}$ and the derived category $\mathcal{D}_{dg}(\mathcal{D}).$ There is a Yoneda embedding of $\mathcal{D}$ into $\mathcal{C}_{dg}(\mathcal{D}).$ One says that $\mathcal{D}$ is {\it pretriangulated} if the image of the Yoneda embedding is closed under taking cones of morphisms and under the shift functor. One can show that $\mathcal{D}$ is pretriangulated if and only if its underlying category $Z^0(\mathcal{D})$ is Frobenius; in such a case the stable category of the latter is the homotopy category $H^0(\mathcal{D}).$ Then the Yoneda embedding induces an embedding of $H^0(\mathcal{D})$ into the full subcategory $\per(\mathcal{D})$ of perfect (and, equivalently, compact) objects in $\mathcal{D}_{dg}(\mathcal{D}).$ This perfect derived category is then the idempotent completion of $H^0(\mathcal{D}).$ Thus, if $H^0(\mathcal{D})$ is idempotent complete, then we have an equivalence  $H^0(\mathcal{D}) \overset\sim\to \per(\mathcal{D}).$ When this condition holds, one says that $\mathcal{D}$ is a {\it triangulated DG-category}, or that $\mathcal{D}$ is {\it saturated}, or {\it Morita fibrant}. The latter notion reflects the fact that there exists a model structure on the category of DG-categories, s.t. triangulated DG-categories are precisely the fibrant objects. This model structure is called the {\it Morita model structure}. It is known that each DG category has a Morita fibrant replacement, i.e. that it is Morita equivalent to a  triangulated DG-category, cf. \cite{T2}. 

For a triangulated category $\mathcal{T},$ by an {\it enhancement} one understands a triangulated DG-category $\mathcal{D}$ with a triangulated equivalence $\mathcal{T} \overset\sim\to \per(\mathcal{D}).$ It is known that an idempotent complete triangulated category is algebraic if and only if it has an enhancement: if for $\mathcal{T}$ there exists a Frobenius category $\mathcal{F},$ s.t. $\mathcal{T} \overset\sim\to \underline{\mathcal{F}},$ then one can endow the category of the complexes with projective-injective components in $\mathcal{P(F)}$ with a natural DG-category structure, such that the corresponding homotopy category $H^0(\mathcal{C}_{dg}(\mathcal{P(F)}))$ will be equivalent to $\mathcal{T}.$ If $\mathcal{D}$ is pretriangulated and a Frobenius category $Z^0(\mathcal{D})$ is idempotent complete, then its stable category $H^0(\mathcal{D})$ is idempotent complete as well, and $\mathcal{D}$ is triangulated. We can apply our main theorems to the case of such DG-categories.

\begin{corollary} 
Let $\mathcal{D}, \mathcal{D}'$ be a pair of  triangulated DG-categories whose underlying categories $Z^0(\mathcal{D}), Z^0(\mathcal{D}')$ satisfy conditions (C1) and (C2). If a DG-functor $F: \mathcal{D}' \to \mathcal{D}$ induces a fully faithful map of the underlying categories, then it induces an embedding of algebras 
$$F_\ast: \mathcal{SDH}(Z^0(\mathcal{D}')), \mathcal{P}((Z^0(\mathcal{D}'))) \hookrightarrow \mathcal{SDH}(Z^0(\mathcal{D})), \mathcal{P}((Z^0(\mathcal{D}))).$$
If, moreover, $F$ induces an equivalence of perfect derived categories $\per(\mathcal{D}') \overset\sim\to \per(\mathcal{D})$ and an isomorphism of the Grothendieck groups $K_0(\mathcal{P}((Z^0(\mathcal{D}')) \overset\sim\to \mathcal{P}((Z^0(\mathcal{D})),$ then $F_\ast$ is an isomorphism of algebras.
\end{corollary}
\end{example}

\section{Further directions}

Unlike the derived Hall algebras, the semi-derived Hall algebras depend not only on the a triangulated category $\underline{\mathcal{F}},$ but also on some additional amount of information (concerning the Grothendieck group of projective-injectives in $\mathcal{P(F)}$). Let us explain why this is natural to expect for Hall algebras related to triangulated categories. In their recent work, Dyckerhof and Kapranov \cite{DK} showed that the object defining the Hall algebra is not a category itself but rather its {\it Waldhausen S-space}. The associativity of the Hall algebra follows from a property of the S-space that Dyckerhof and Kapranov call being {\it 2-Segal}. Waldhausen spaces were introduced in order to define the algebraic K-theory for an appropriate class of categories. It is now well-known that the K-theory is not an invariant of triangulated categories with respect to triangle equivalences. This is one of unsatisfying facts concerning triangulated categories, aside, e.g., the non-functoriality of the cone. Therefore, one has to consider some ``enhancement'' of triangulated categories to work with. There are several closely related approaches to this problem: DG-categories, stable infinity categories, derivateurs, model categories, Frobenius pairs. To any of them one can associate a Waldhausen-like construction. It is shown by Schlichting that any map of Frobenius pairs inducing an equivalence of the associated derived categories induce a homotopy equivalence of corresponding  Waldhausen spaces as well. It is therefore natural to expect that a good notion of the Hall algebra for triangulated categories should give an invariant under maps between Frobenius pairs inducing equivalences of their derived categories, but not necessarily an invariant under all equivalences of triangulated categories. As any category equivalent to the derived category of a Frobenius pair is also algebraic, our construct and main theorems can be thought of as a step on the way to construct such a notion. Indeed, the class of categories under consideration is the same, what changes is the class of functors between them. One should say also that it is very often useful to realize a triangulated category not as a stable category but as a derived category of a Frobenius pair. This is true, e.g., for derived categories of exact and DG-categories, for singularity categories and for (generalized) cluster categories. All the same can be said about DG-enhancement. 
We will give a construction of Hall algebras of Frobenius pairs and of DG-categories and their DG-quotients (in some generality) in the upcoming sequel of this work. 

We should say also that the K-theory is a well-defined invariant of a certain strictified version of triangulated categories, called {\it Heller,} or {\it $\infty-$triangulated categories}, see \cite{Bal}\cite{Kun07}\cite{M}. This class covers all algebraic triangulated categories. It is therefore possible that one can define Hall algebras for $\infty-$triangulated categories, invariant under equivalences between them, with the multiplication involving higher triangles.

As it was mentioned in the Introduction, we hope that one can define geometric, motivic or even cohomological counterparts of the semi-derived Hall algebras introduced in this paper.

\end{document}